\documentclass[12pt]{amsart}
\usepackage[top=30truemm,bottom=30truemm,left=25truemm,right=25truemm]{geometry}
\usepackage{mathrsfs}

\usepackage{bm}
\usepackage{amsfonts,amssymb}
\usepackage{dsfont}
\usepackage{extarrows}
\usepackage{amsmath}
\usepackage{mathrsfs}
\usepackage{enumerate}
\usepackage{amscd}
\usepackage[all]{xy}
\usepackage{hyperref}
\theoremstyle{plain} 

\theoremstyle{definition} 

\newtheorem{thm}{Theorem}[section]

\newtheorem{lem}[thm]{Lemma}
\newtheorem{prop}[thm]{Proposition}

\theoremstyle{definition}
\newtheorem{defn}{Definition}[section]

\theoremstyle{remark}

\newcommand{\half}{\frac{1}{2}}

\newcommand{\be}{\begin{equation}}
\newcommand{\ee}{\end{equation}}
\newcommand{\bea}{\begin{eqnarray}}
\newcommand{\eea}{\end{eqnarray}}
\newcommand{\ben}{\begin{eqnarray*}}
	\newcommand{\een}{\end{eqnarray*}}
\newcommand{\bt}{\begin{split}}
	\newcommand{\et}{\end{split}}
\newcommand{\bet}{\begin{equation}}

\newcommand{\mr}{\mathbb{R}}
\newcommand{\ra}{\rightarrow}
\newcommand{\beq}{\begin{equation*}}
\newcommand{\eeq}{\end{equation*}}

\newcommand{\pa}{\partial}
\newcommand{\bp}{\bar{\partial}}

\begin{document}

\title[Characterization of Curvature positivity]
{Characterization of Curvature positivity of Riemannian metrics on flat vector bundles}

\author[F. Deng]{Fusheng Deng}
\address{Fusheng Deng: \ School of Mathematical Sciences, University of Chinese Academy of Sciences\\ Beijing 100049, P. R. China}
\email{fshdeng@ucas.ac.cn}

\author[X.Zhang]{Xujun Zhang}
\address{Xujun Zhang: \ School of Mathematical Sciences, University of Chinese Academy of Sciences\\ Beijing 100049, P. R. China}
\email{zhangxujun16@mails.ucas.ac.cn}

\begin{abstract}
We give a characterization of Nakano positivity of Riemannian flat vector bundles over bounded domains $D\subset\mr^n$ in terms of solvability
of the $d$ equation with certain good $L^2$ estimate condition.
As an application, we give an alternative proof of the matrix-valued Prekopa's theorem that is originally proved by Raufi.
Our methods are inspired by the recent works of Deng-Ning-Wang-Zhou on characterization of Nakano positivity of Hermitian holomorphic vector
bundles and positivity of direct image sheaves associated to holomorphic fibrations.
\end{abstract}

\maketitle

\section{Introduction}\label{sec:intro}
Recently in \cite{DNW19}, Deng-Ning-Wang introduce a new concept-the optimal $L^2$ estimate condition-for plurisubharmonic functions.
They prove that a $C^2$ function must be plurisubharmonic if it satisfies the optimal $L^2$ estimate condition,
which means that plursubharmonic functions are the only choice for weights in H\"ormander's $L^2$ estimate for the $\bar\partial$ equation.
This result can be roughly viewed as the converse of H\"ormander's $L^2$ theory for $\bar\partial$.
As a continuation of this work, Deng-Ning-Wang-Zhou in \cite{DNWZ20} find a characterization of Nakano positivity for Hermitian holomorphic vector bundles,
which is applied to give a simple and transparent proof of Berndtsson's fundamental result on the Nakano positivity of certain direct image sheaves \cite{Ber09}.

Inspired by the above mentioned works, the purpose of the present paper is to establish the parallel results for convex functions and for curvature of Riemannnian metrics on flat vector bundles.
Before stating the main results, we first introduce some notions.
\begin{defn}\label{def:optimal $L^2$ estimate properties for function}
Let $\phi$ be a locally integrable real valued function on a domain $D\subset\mr^n$.
We say that $\phi$ satisfies \emph{the optimal $d-L^2$ estimate property}, if for any $d$-closed smooth $1$-form $f=\sum^{n}_{i=1} f_{i}dx_i$ with compact support on $D$,
and any smooth strictly convex function $\psi$ on $D$, the equation $du=f$ can be solved on $D$ with the estimate
 $$
 \int_{D} |u|^2 e^{-\phi-\psi} dx \leq \int_{D}\sum^{n}_{i,j=1}\psi^{ij}f_{i}f_{j}e^{-\phi-\psi} dx,
 $$
 where  $dx$ is the Lebesgue measure on $\mr^n$  and $(\psi^{ij})_{n\times n}$ stands for the inverse of the matrix $(\psi_{ij})_{n\times n}$, with $\psi_{ij}=\frac{\pa^2 \psi}{\pa x_{i} \pa x_{j}}$.
\end{defn}

The first main result is the following

\begin{thm}\label{thm:char convexity}
Let $\phi:D\ra\mr$ be a $C^2$ function on a domain $D\subset\mr^n$.
If $\phi$ satisfies the optimal $d-L^2$ estimate condition, then $\phi$ must be a convex function.
\end{thm}

In Theorem \ref{thm:char convexity}, the convexity of $\phi$ means that the Hessian of $\phi$ is semi-positive everywhere.
The fact that a convex function on $\mr^n$ satisfies the optimal $d-L^2$ condition was proved in \cite{Bra-Lie76}.
So Theorem \ref{thm:char convexity} can be understood as a converse of $L^2$-estimate for the $d$-equation in \cite{Bra-Lie76}.
Theorem \ref{thm:char convexity} and its proof explains why the weight functions in the $L^2$ estimate of the $d$-equation should be convex,
and hence answers a question posed by Berndtsson in \cite{Ber07}("It is also not clear why the weight function should be required to be convex" in Introduction in \cite{Ber07}).

Applying Theorem \ref{thm:char convexity}, we can give a very simple and transparent proof of Prekopa's theorem for convex functions.

\begin{thm}[\cite{Pre73}]\label{thm:prekopa}
Assume that $\tilde \phi(x,y)$ is a convex function on $\mr^n_x\times\mr^m_y$. Then the function $\phi(x)$ defined by
$$e^{-\phi(x)}:=\int_{\mr^m}e^{-\tilde \phi(x,y)}dy$$
is a convex function on $\mr^n$.
\end{thm}

Given a function $\phi$ on a domain $D\subset\mr^n$, we can view $h:=e^{-\phi}$ as a Riemannian metric on the trivial line bundle $\pi:D\times\mr\ra\mr^n$ over $D$,
where $\pi$ is the natural projection. In analogue to the case of Hermitian holomorphic line bundle, we understand that $h$ has semi-positive curvature if $\phi$ is convex.
Motivated by this observation and the curvature operator defined for holomorphic vector bundles,
Raufi introduces the notion of Nakano positivity (phrased as log concavity in the sense of Nakano) for Riemannian metrics on trivial vector bundles.

Let $D$ be a domain in $\mr^n$ and let $p: E=D\times\mr^r\ra D$ be the trivial vector bundle over $D$ of rank $r$.
Then a Riemannian metric on $E$ can be naturally identified with a map $g:D\ra Sym_r(\mr)^+$, where $Sym_r(\mr)^+$ is the space of positive definite matrices of order $r$.
Given such a map $g$, the corresponding inner product on the fiber of $E$ over $x\in D$ is given by
$$\langle u,v \rangle_g=\langle g(x)u, v\rangle=v^Tg(x)u, \ u,v\in\mr^r,$$
where $v^T=(v_1,\cdots, v_r)$ is the transpose of $v$.

\begin{defn}{(\cite{Rau13})}\label{def:log-concave-matrix}
Let $g: D \ra  Sym_{r}(\mr)^{+}$ be a $C^2$ Riemannian metric on $E$. Set
$$
\theta^{g}_{jk}=-\frac{\pa}{\pa x_{k}}(g^{-1}\frac{\pa g}{\pa x_j}), \forall 1\leq j,k \leq n,
$$
here the differentiation should be interpreted elementwise. We say that $(E, g)$ is Nakano positive (semipositive) if we have
$$
\sum^{n}_{j,k=1} \langle \theta^{g}_{jk}(x)u_{j},u_{k} \rangle _{g}>0\ (\geq 0)
$$
holds for any $x\in D$ and any  $n$-tuple of vectors $\left\{u_{j} \right\}^{n}_{j=1} \subset \mr^r$. If we set $\Theta^{g}=[\theta^{g}_{jk}]$, then the left side of the above inequality can be written as
$$
\langle \Theta^{g}u,u \rangle_{g}= \sum^{n}_{j,k=1} \langle \theta^{g}_{jk}(x)u_{j},u_{k} \rangle _{g}.
$$
\end{defn}

In Definition \ref{def:log-concave-matrix}, if $r=1$ and writing $g=e^{-\phi}$,
then $(E,g)$ is Nakano semipositive if and only if $\phi$ is a convex function.
Therefore, Nakano positivity can be viewed as a generalization of convexity.

Let $D$ and $E$ as above.
We denote by $\Lambda^p(D,E)$ the space of smooth $p$-forms on $D$ with values in $E$.
Formally an element $f\in \Lambda^p(D,E)$ can be written as
$$f=\sum_{1\leq i_1<\cdots<i_p\leq n}f_{i_1\dots i_p}dx_{i_1}\wedge\cdots\wedge dx_{i_p},$$
where $f_{i_1\cdots i_p}:D\ra\mr^r$ are smooth maps that are identified with smooth sections of $E$.
We can define the exterior differential $d:\Lambda^p(D,E)\ra \Lambda^{p+1}(D,E)$ in the natural way.
For any $f\in \Lambda^{p-1}(D,E)$, we have $d(df)=0\in \Lambda^{p+1}(D,E)$.

If $g$ is a Riemannian metric on $E$, and $f, f'\in \Lambda^1(D,E)$ are given by
$$f=\sum_{j=1}^n f_jdx_j, f'=\sum_{j=1}^nf'_jdx_j,$$
then the pointwise inner product of $f$ and $f'$ is given by
$$\langle f, f'\rangle_g=\sum^n_{j=1}\langle f_j, f'_j\rangle_g.$$
The inner product of $f$ and $f'$ is defined as
$$\langle\langle f,f'\rangle\rangle_g:=\int_D\langle f,f'\rangle_g dx.$$
The inner product on $\Lambda^p(D,E),\ (p\geq 0)$ is defined similarly.
We denote by $L_p^2(D,E)$ the space of measurable $p$-forms $f$ with values in $E$ such that $\|f\|^2_g:=\langle\langle f,f\rangle\rangle_g<+\infty$.
For simplicity, we denote $L^2_0(D,E)$ by $L^2(D,E)$, which is the space of measurable sections $f$ of $E$ such that $\|f\|^2_g<+\infty$.

We now generalize Definition \ref{def:optimal $L^2$ estimate properties for function} to the following
\begin{defn}\label{def:vector bundle optimal L^2}
Assume $D$ is a domain in $\mr^n$, $E=D \times \mr^r$ is the trivial vector bundle over $D$,
and $g: D \ra  Sym_{r}(\mr)^{+}$ is a $C^2$ Riemannian metric on $E$.
We say that $g$ satisfies \emph{the optimal $d-L^2$ estimate condition} if for any $C^2$ strictly convex function $\psi$ on $D$, and any $d$-closed $f \in \Lambda^1(D,E)$ with compact support,
there exists $u \in L^2 (D, E)$ satisfying $du=f$ and
$$
\int_{D} \langle u,u \rangle  _{g} e^{-\psi} dx \leq  \int_{D}\langle (Hess\; \psi)^{-1}f,f \rangle  _{g} e^{-\psi} dx.
$$

\end{defn}

Our second main result is the following.

\begin{thm}\label{thm:cha nakano positive}
Assume $D$ is a domain in $\mr^n$, $E=D \times \mr^r$ is the trivial vector bundle over $D$,
and $g: D \ra  Sym_{r}(\mr)^{+}$ is a $C^2$ Riemannian metric on $E$.
If $g$ satisfies \emph{the optimal $d-L^2$ estimate condition}, then $(E, g)$ is Nakano semipositive.
\end{thm}

Theorem 1.3 may be viewed as the converse of the following Theorem \ref{thm:L^2 estimate for vector bundle} for the $L^2$-estimate of $d$-equation,
which is a real analysis analogue of Theorem 1.1 in \cite{DNWZ20} for the $\bar\partial$ equation.
If $r=1$, then Theorem \ref{thm:cha nakano positive} reduces to Theorem \ref{thm:char convexity}.

One of the main ingredients in the proof of Theorem \ref{thm:cha nakano positive} is a Bochner type identity.

\begin{prop}\label{prop:Bochner type identity}
Assume $D$ is a domain in $\mr^n$, $E=D \times \mr^r$ is the trivial vector bundle defined over $D$, and $g: D \ra  Sym_{r}(\mr)^{+}$ is a $C^2$ Riemannian metric on $E$.
Then for any $\alpha =\sum^{n}_{i=1} \alpha_{i} dx_{i}\in \Lambda^1(D, E)$ with compact support, we have
\begin{equation*}\label{eq:Bochner-type-identity-matrix}
\int_{D} \langle  d^*\alpha,d^* \alpha \rangle _{g}dx+\int_{D} \langle  d\alpha,d \alpha \rangle _{g}dx=\int_{D}\sum^{n}_{i,j=1}\langle  \theta^{g}_{ij}\alpha_{i},\alpha_{j} \rangle  _{g}dx+\int_{D} \sum^{n}_{i,j=1} \langle   \frac{\pa \alpha_{i}}{\pa x_{j}}, \frac{\pa \alpha_{i}}{\pa x_{j}} \rangle _{g} dx,
\end{equation*}
where $d^*$ is the formal adjoint operator of $d:L^2(D,E)\ra L^2_1(D,E)$.
\end{prop}

In the case that $r=1$, the identity in Proposition \ref{prop:Bochner type identity} can be found in \cite{Ber07}.

On the other hand, it is known that a Nakano semipositive Riemannian trivial bundle over $\mr^n$ satisfies the optimal $L^2$-estimate condition.
The following result is an analogue of H\"ormander's $L^2$-estimate for the $\bar\partial$-equantion \cite{Hor65}.

\begin{thm}[{\cite[Theorem 4]{Cor19}}]\label{thm:L^2 estimate for vector bundle}
Let $g: \mr^n \ra Sym_{r}(\mr)^{+}$ be a Riemannian metric on the trivial vector bundle $p: E=\mr^n\times\mr^r\ra\mr^n$.
Assume that $(E,g)$ is Nakano semipositive and $\int_{\mr^n} |g|< + \infty$.
Then for any $d$-closed $f\in \Lambda^1(\mr^n, E)$, the equation $du=f$ can be solved with $u\in L^2(\mr^n, E)$ satisfying the following estimate
$$\int_{\mr^n}|u|^2_g\leq \int_{\mr^n}\langle(\Theta^g(x))^{-1}f, f\rangle_g dx.$$
\end{thm}

Indeed, Theorem \ref{thm:L^2 estimate for vector bundle} is a slight reformulation of Theorem 4 in \cite{Cor19}.
In the case that $r=1$, Theorem \ref{thm:L^2 estimate for vector bundle} is due to Brascamp and Lieb \cite{Bra-Lie76}.
Theorem \ref{thm:L^2 estimate for vector bundle}  is proved in \cite{Cor19} based on a Bochner-type identity ({\cite[Fact 8]{Cor19}}) for smooth sections of $E$.
We remark that based on H\"ormander's idea for $L^2$-estimate of $\bar\partial$ \cite{Hor65},
Theorem \ref{thm:L^2 estimate for vector bundle} can also be deduced from Proposition \ref{prop:Bochner type identity},
combing with the following generalized Cauchy-Schwarz inequality:
$$
\langle \langle f,\alpha \rangle \rangle_g^2 \leq \langle \langle (\Theta^{g})^{-1} f,f \rangle \rangle_g \cdot \langle \langle \Theta^{g} \alpha, \alpha \rangle \rangle_g
$$
where $f,\alpha \in \Lambda^{1}(\mr^n,E)$. Here we omit the details.


A direct consequence of the combination of Theorem \ref{thm:cha nakano positive} and  Theorem \ref{thm:L^2 estimate for vector bundle}
is a matrix-valued version of Theorem \ref{thm:prekopa}.

\begin{thm}{(\cite{Rau13})}\label{thm:Prekopa-matrix-valued}
Let $\tilde g(x,y):\mr^{n}_{x} \times \mr^{m}_{y} \ra Sym_{r}(\mr)^+$ be a $C^2$ Riemannnian metric on the trivial bundle $\tilde E=(\mr^n\times\mr^m)\times\mr^r\ra \mr^n\times\mr^m$ over $\mr^n\times\mr^m$.
Define
$$
g(x)=\int_{\mr^{m}} \tilde g(x,y) dy \in Sym_{r}(\mr)^+.
$$
If $(\tilde E,\tilde g)$ is Nakano semipositve and $g$ is $C^2$ smooth, then $g$ is Nakano semipositive, viewed as a Riemannian metric on the trivial bundle $E=\mr^n\times\mr^r$ over $\mr^n$.
\end{thm}

Theorem \ref{thm:Prekopa-matrix-valued} was originally proved by Raufi in \cite{Rau13}.
Raufi's proof contains two main ingredients:
Berndtsson's method to the positivity of direct image bundles \cite{Ber09} and a Fourier transform technique.
To avoid the Fourier transform technique and complex analysis in the proof,
Cordero-Erausquin recently produced a new proof of Theorem \ref{thm:Prekopa-matrix-valued} based on Theorem \ref{thm:L^2 estimate for vector bundle},
which is motivated by the proof of Theorem \ref{thm:prekopa} given by Brascamp and Lieb \cite{Bra-Lie76}.
Our method to Theorem \ref{thm:prekopa} and Theorem \ref{thm:Prekopa-matrix-valued} is inspired by the works in \cite{DNW19} and \cite{DNWZ20}, .
and is essentially different from those in \cite{Rau13} and \cite{Cor19}.
The main idea in the proof of Theorem \ref{thm:Prekopa-matrix-valued} is as follows.
Given that $(\tilde E,\tilde g)$ is Nakano semipositive, then $(\tilde E,\tilde g)$ satisfies the optimal $d-L^2$ estimate condition by Theorem \ref{thm:L^2 estimate for vector bundle}.
By Fubini theorem, it is easy to show that $(E, g)$ also satisfies the optimal $d-L^2$ estimate condition,
and hence is Nakano semipositive by Theorem \ref{thm:cha nakano positive}.
In the recent preprint \cite{DHJ20}, Theorem \ref{thm:Prekopa-matrix-valued} is proved and generalized by a related but different method that is
based on the characterization of Nakano positivity of Hermitian holomorphic vector bundles in \cite{DNWZ20} and a group action technique.

All the above results also holds for Hemitian metrics on complex vector bundles.
Recall that a vector bundle $E$ over a manifold $M$ is called flat if the pull back bundle $\pi^*E$ over $\tilde M$ is trivial,
where $\pi:\tilde M\ra M$ is the universal covering of $M$.
In a forthcoming work, we will generalize the above results to flat vector bundles with Riemannian metrics over general Riemannian manifolds.

Though Theorem \ref{thm:char convexity} and Theorem \ref{thm:prekopa} are special cases of
Theorem \ref{thm:cha nakano positive} and Theorem \ref{thm:Prekopa-matrix-valued} respectively,
we still present their proofs here since their proofs are simpler than those of the later ones in technique
and hence can help the readers to grasp the main ideas.
The remaining of the paper is organized as follows.
In \S \ref{sec:The convexity of weight function}, we give the proof of Theorem \ref{thm:char convexity}, and in \S \ref{sec:prekopa} give the proof of
Theorem \ref{thm:prekopa}. We then establish the basic Bochner type identity (Property \ref{prop:Bochner type identity}) in \S \ref{sec:Bochner type identity},
and prove Theorem \ref{thm:cha nakano positive} in \S \ref{sec:cha of Nakano positivity} and prove Theorem \ref{thm:Prekopa-matrix-valued} in the final \S \ref{sec:matrix valued prekopa}.

$\mathbf{Acknowlegements.}$
The methods to the main results in the present paper are strongly inspired by the works in \cite{DNW19} and \cite{DNWZ20}.
The first author are grateful to Professor Jiafu Ning, Zhiwei Wang, and Xiangyu Zhou for related collaborations and discussions.
The authors are partially supported by the NSFC grant 11871451.

\section{Characterization of convex functions in terms of $L^2$ estimate for the $d$-equation}\label{sec:The convexity of weight function}
The aim of this section is to prove Theorem \ref{thm:char convexity}.

Assume $D$ is a domain in $\mr^n$ and $\phi:D \ra \mr$ is a $C^2$ smooth function,
we define $L^2(D,e^{-\phi})$ to be the Hilbert space of (real valued) measurable functions $f$ on $D$ such that
$$\|f\|^2_\phi:=\int_D|f|^2e^{-\phi}dx<\infty,$$
where $dx$ is the Lebesque measure on $\mr^n$.
The weighted inner product on $L^2(D,e^{-\phi})$ is defined in the the following natural way
$$
(f,g)_{\phi}=\int_{D} f\cdot g e^{-\phi} dx.
$$
Furthermore, for a measurable $1$-form $\alpha=\sum^{n}_{j=1} \alpha_{i}dx_i$ defined on $D$, the  weighted $L^2$ norm of $\alpha$ is defined by
$$
\|\alpha\|_{\phi}^2 =  \sum^n_{j}\|\alpha_j\|^2_\phi=\sum^n_{j} \int _{D}|\alpha _j|^2e^{-\phi} dx.
$$
We denote by $L^2_1(D,e^{-\phi})$ the Hilbert space of measurable $1$-forms on $D$ with finite norm.
The inner product on $L^2_1(D, e^{-\phi})$ is given by
$$(\alpha,\beta)_\phi:=\sum^n_{j=1}(\alpha_j,\beta_j)_\phi=\sum^n_{j=1}\int_D\alpha_j\beta_je^{-\phi}$$
for $\alpha=\sum^n_{j=1}\alpha_jdx_j,\ \beta=\sum^n_{j=1}\beta_j dx_j$.

A simple calculation shows that the formal adjoint of the densely defined operator $d:L^2(D,e^{-\phi})\ra L^2_1(D, e^{-\phi})$ is give by
$$
\delta _{\phi} \alpha =-\sum^{n}_{j=1} (\frac{\partial\alpha_j}{\partial x_j}-\frac{\partial\phi}{\partial x_j}\cdot\alpha _j ),
$$
where $\alpha$ is any smooth $1$-form on $D$ with compact support.

The following identity is required for the proof.
\begin{lem}[{\cite[Proposition 3.1]{Ber07}}]\label{lem:bochner type identity for line bundle} Let $\alpha=\sum^n_j\alpha_jdx_j$ be a smooth compactly supported $1$-form on $\mr^n$. Then
\begin{equation}
\int _{\mr^n} (\sum^{n}_{j,k=1} \frac{\partial^2\phi}{\partial x_j\partial x_k} \alpha _{j} \alpha _{k} +\sum^{n}_{j,k=1} |\frac{\partial\alpha_j}{\partial x_k} |^2)e^{-\phi} dx=\int _{\mr^n} (|\delta_{\phi } \alpha |^2+|d\alpha |^2)e^{-\phi} dx.
\end{equation}
\end{lem}
\par Now we are ready to prove Theorem \ref{thm:char convexity}, for the convenience, we restate here.

\begin{thm}[=Theorem \ref{thm:char convexity}]
Let $D$ be a domain in $\mr^n$ and $\phi:D\ra\mr$ be a $C^2$ function.
If for any $d$-closed smooth $1$-form $f=\sum^n_{j=1}f_jdx_j$ on $D$ with compact support and any smooth strictly convex function $\psi$ on $D$,
the equation $du=f$ can be solved on $D$ with the estimate
$$
 \int_{D} |u|^2 e^{-\phi-\psi} dx \leq \int_{D}\sum^{n}_{i,j=1}\psi^{ij}f_{i}f_{j}e^{-\phi-\psi} dx,
 $$
 then $\phi$ is a convex function, where  $dx$ is the Lebesgue measure on $\mr^n$  and $(\psi^{ij})_{n\times n}$ stands for the inverse of the matrix $(\psi_{ij})_{n\times n}$, with $\psi_{ij}=\frac{\pa^2 \psi}{\pa x_{i} \pa x_{j}}$.
\end{thm}

\begin{proof}
With Lemma \ref{lem:bochner type identity for line bundle} in hand, we can now follow the idea of the proof of Theorem 1.1 in \cite{DNW19}.

Let $\psi$ be any smooth strictly convex function on $D$. By assumption,
the equation $du=f$ on $D$ can be solved for any compactly supported $d$-closed $1$-form $f=\sum^n_{j=1}f_jdx_j$ with the estimate:
$$
\int _{D} |u|^2 e^{-\phi-\psi} dx\leq \int _{D} \sum^{n}_{j,k=1} \psi ^{jk}f_j f_k e^{-\phi-\psi} dx.
$$
\par For any $d$-closed $1$-form $\alpha$ with compact support we have:
$$
\begin{aligned}
|(\alpha ,f )_{\phi+\psi}| & = |(\alpha, du )_{\phi+\psi}|=|(\delta_{\phi+\psi} \alpha, u)_{\phi+\psi}| \\
&\leq \|u\|_{\phi+\psi} \cdot  \|\delta_{\phi+\psi} \alpha \|_{\phi+\psi}.
\end{aligned}
$$
Combing this with Lemma \ref{lem:bochner type identity for line bundle}, we have
\begin{equation}\label{ineq:3}
|(\alpha ,f )_{\phi+\psi}|^2
\leq \int_{D} \sum^{n}_{j,k=1} \psi ^{jk} f_j f_{k} e^{-\phi-\psi}  dx\times \int_{D} \sum^{n}_{j,k=1} \left((\phi_{jk}+\psi_{jk}) \alpha _{j} \alpha _{k} +|\frac{\partial \alpha_j}{\partial x_k}|^2\right)e^{-\phi-\psi}dx.
\end{equation}
Since the inequality (\ref{ineq:3}) holds for any $\alpha$ with compact support, setting
$$
(\alpha _1,...,\alpha _n)=(f_1,...,f_n)(\psi ^{jk}),
$$
then the left hand in \eqref{ineq:3} becomes
$$\left(\sum^n_{j,k=1}\int_D \psi^{jk}f_jf_ke^{-\phi-\psi}\right)^2,$$
which is also equal to
$$\left(\sum^n_{j,k=1}\int_D \psi_{jk}\alpha_j\alpha_ke^{-\phi-\psi}\right)^2.$$
We therefore obtain form  \eqref{ineq:3} the following inequality:
\begin{equation}\label{1}
\int _D \sum^{n}_{j,k=1} \phi_{jk} \alpha _{j} \alpha _{k}e^{-\phi-\psi} +\int_D\sum^{n}_{j,k=1} |\frac{\partial \alpha_j}{\partial x_k} |^2e^{-\phi-\psi} dx\geq 0.
\end{equation}

The next thing to do is to argue by contradiction.

Suppose $\phi$ is not convex,
then there exists $x_0 \in D$, $r>0$, a constant $c>0$, and $\zeta =(\zeta_1,...,\zeta_n )\in \mr^n $ with $|\zeta|=1$ such that
$$
\sum^{n}_{j,k=1} \phi _{jk}(x)\zeta_{j}\zeta_{k}<-c
$$
holds for any $x\in B(x_0,r)\subset D $, where $B(x_0, r)$ is the ball centered at $x_0$ with radius $r$.

For the simplicity, we may assume $x_0=0$ and write $B(0,r)$ as $B_r$.
The purpose of the following construction is to show that the inequality (\ref{1}) doesn't hold for some special $\alpha$.

Since we can solve the equation $du=f$ for any $d$-closed smooth $1$-form $f$. We choose $f =dv$ with
$$
v(x)=(\sum^{n}_{j=1} \zeta _j x_j )\chi(x),
$$
viewed as a smooth function on $D$,
where $\chi \in C^{\infty }_{c}(B_r)$ satisfying $\chi (x)=1$ for $x\in B_{\frac{r}{2}}$.
Then
$$
f(x)=\sum^{n}_{j=1} \zeta _j dx_j
$$
for
$x\in B_{\frac{r}{2} }$. For $s>0$, we set
$$
\psi _s (x)=s(|x|^2-\frac{r^2}{4} ),
$$
which is a strictly convex function on $\mr^n$ with
$$
\frac{\pa^2 }{\pa x_{j} \pa x_{k}}\psi _s (x) =2s \delta_{jk}.
$$
Set
$$
(\alpha^s _1,...,\alpha^s _n )=(f_1,...,f_n )(\psi^{jk}  _s)=\frac{1}{2s}(f_1,...,f_n ),
$$
then on $B_{\frac{r}{2}}$ we have
$$
\alpha ^s(x)=\frac{1}{2s}\sum^{n}_{j=1} \zeta _j dx_j
$$
and $\frac{\pa \alpha ^s_j}{\pa x_k}\equiv 0,\ j,k=1,...,n$.
Since $f$ has compact support, there is a constant $C>0$ such that $|\alpha ^s_{j}|\leq \frac{C}{s}$ and $|\frac{\pa \alpha ^s_j}{\pa x_k}| \leq \frac{C}{s}$ hold for any $x\in D,j,k=1,...,n$ and any $s>0$.

Replacing $\alpha $ and $\psi$ in the left hand of (\ref{1}) by $\alpha^s $ and $\psi_s$ defined as above and multiplying it by $s^2$, we get
\begin{equation}
\begin{aligned}
&s^2 \int _D \sum^{n}_{j,k=1} \phi_{jk} \alpha^s _{j} \alpha^s _{k}e^{-\phi-\psi_{s}} dx +s^2 \int _D \sum^{n}_{j,k=1} |\frac{\partial\alpha^s_j}{\partial x_k} |^2e^{-\phi-\psi_{s}} dx\\
=&  s^2 (\int _{B_{\frac{r}{2}}}+\int _{D\backslash B_{\frac{r}{2}} })\sum ^{n}_{j,k=1} \phi _{ij}  \alpha ^s_j \alpha ^s_{k}e^{-\phi-\psi_s } dx+s^2 \int _D \sum^{n}_{j,k=1} |\frac{\partial\alpha^s_j}{\partial x_k} |^2e^{-\phi-\psi_{s}} dx\\
\leq & -\frac{c}{4} \int _{B_{\frac{r}{2}}}e^{-\phi-\psi_s } dx+s^{2}\int _{D\backslash B_{\frac{r}{2}} }\sum ^{n}_{j,k=1} \phi _{ij}  \alpha ^s_j \alpha ^s_{k}e^{-\phi-\psi_s  } dx+s^2 \int _{D} \sum^{n}_{j,k=1} |\frac{\partial\alpha^s_j}{\partial x_k} |^2e^{-\phi-\psi_{s}} dx.
\end{aligned}
\end{equation}
Note that $\psi _s\nearrow +\infty$ on $D\backslash B_{\frac{r}{2}} $ as $s\ra +\infty$ and $|\alpha ^s _{j}|\leq \frac{C}{s} $ holds for any $s>0$, we get
$$
\underset{s \ra +\infty}{lim}s^{2}\int _{D\backslash B_{\frac{r}{2}} }\sum ^{n}_{j,k=1}\phi_{ij} \alpha ^s_j \alpha ^s_{k}e^{-\phi-\psi_s }dx =0.
$$
Since $\frac{\pa \alpha_j}{\pa x_k}(x)=0$ for $x\in B_{\frac{r}{2}}, j,k=1,...,n$, and $|\frac{\pa \alpha ^s _j}{\pa x_k } |\leq \frac{C}{s}$ on $D$ for $j,k=1,...,n$ and $s>0$, we get
$$
s^2 \int _D \sum ^{n}_{j,k=1} |\frac{\pa \alpha^s _j}{\pa x_k} |^2e^{-\phi-\psi_s}dx=s^2 \int _{D\backslash B_{\frac{r}{2}}} \sum ^{n}_{j,k=1} |\frac{\pa \alpha^s _j}{\pa x_k} |^2e^{-\phi-\psi_s}dx \ra 0, s \ra +\infty.
$$
Thus for sufficiently large $s$ we have
$$
\int _D (\sum^{n}_{j,k=1} \phi_{jk} \alpha^s _{j} \alpha^s _{k} +\sum^{n}_{j,k=1} |\frac{\pa \alpha_j}{\pa x_k}|^2)e^{-(\phi+\psi_s)}dx < 0.
$$
This leads to a contradiction to the inequality (\ref{1}). Thus $\phi$ is a convex function.
\end{proof}

\section{A new method to the classical Prekopa's Theorem}\label{sec:prekopa}
\par In this section we give an alternative proof of the classical Prekopa's theorem.
Our method is based on Theorem \ref{thm:char convexity}.
\begin{thm}[=Theorem \ref{thm:prekopa}]
Assume that $\tilde \phi(x,y)$ is a convex function on $\mr^n_x\times\mr^m_y$. Then the function $\phi(x)$ defined by
$$e^{-\phi(x)}:=\int_{\mr^m}e^{-\tilde \phi(x,y)}dy$$
is a convex function on $\mr^n$.
\end{thm}

\begin{proof}
 By the standard smoothing procedure, we can assume that $\tilde\phi$ is smooth.
 We can also construct a smooth convex function $h$ on $\mr^n$ such that $e^{\tilde\phi+\epsilon h}$ is integrable on $\mr^n\times\mr^m$ for any $\epsilon>0$.
 So in the proof we may assume that $\tilde\phi$ is smooth and $\int_D e^{-\tilde\phi}<\infty$.
 By Theorem \ref{thm:char convexity}, the remaining is to prove that $\phi$ satisfies the optimal $L^2$ estimate condition on $\mr^m$.

 Let $f=\sum^m_{j=1}f_j(x)dx_j$ be a compactly supported $d$-closed 1-form on $\mr^n$.
 We can also naturally view $f$ as a $d$-closed 1-form on $\mr^n\times\mr^m$ which will be denoted by $\tilde f=\sum^{n}_{j=1}\tilde f_jdx_j$.
 Let $\psi(x)$ be an arbitrary smooth strictly convex function on $\mr^n$.
 We also view $\psi$ as a smooth function on $\mr^n\times\mr^m$ which will be denoted by $\tilde \psi$.

According to Theorem \ref{thm:L^2 estimate for vector bundle}, the equation $d\tilde u=\tilde f$ can be solved on $\mr^n\times\mr^m$ with the estimate
$$
\int_{\mr _x ^n \times \mr_y ^m}|\tilde u|^2 e^{-\tilde \phi-\tilde \psi} dxdy\leq \int_{\mr _x ^n \times \mr_y ^m} \sum_{i,j=1}^n \tilde \psi ^{ij}\tilde f_i \tilde f_j e^{-\tilde \phi-\tilde \psi}dxdy,
$$
where $(\tilde\psi^{ij}(x,y))_{n\times n}$ is the inverse of the matrix $(\frac{\partial^2\tilde\psi}{\partial x_i\partial x_j})_{n\times n}$.

From the identity $d\tilde u=\tilde f$, we see that $\tilde u$ is independent of $y_1,\cdots, y_m$,
hence we can view $\tilde u$ as a function on $\mr^n_x$, which will be denoted by $u$.
It is clear that $du=f$.
By Fubini's theorem, the above inequality becomes to
$$
\int _{\mr^n }|u|^2 e^{-\phi-\psi} dx \leq \int _{\mr^m }\sum_{i,j}  \psi ^{ij} f_i f_j  e^{-\phi-\psi}dx,
$$
which implies, by Theorem \ref{thm:char convexity}, $\psi$ is a convex function on $\mr^n$.
\end{proof}

\section{A Bochner-type identity for 1-forms with values in a vector bundle}\label{sec:Bochner type identity}
Let $D\subset\mr^n$ be a domain and $E=\mr^n \times \mr^r\ra\mr^n$ be the trivial vector bundle of rank $r$ over $\mr^n$.
As in the introduction section, we denote by $\Lambda^p(D, E)$ the space of smooth $p$-forms with values in $E$, for $p\geq 0$.
If $g$ is a Riemannian metric on $E$, then we will denote by $L^2_p(D,E)$ the space of square integrable $p$-forms on $D$ with values in $E$.
For simplicity, $L^2_0(D, E)$ is just denoted by $L^2(D, E)$.
The inner product on $L^2_p(D, E)$ is defined as in the introduction.

We consider the following chain of weight Hilbert spaces
$$
L^{2} (\mr^n, E) \xrightarrow{d} L^{2}_{1} (\mr^n, E) \xrightarrow{d}  L^{2}_{2} (\mr^n, E),
$$
and study the equation
$$
du=f
$$
for $f \in L^{2} _{1}(\mr^n, E) $ with $df=0$.

 The identity $du=f$, in the sense of distributions, means that
$$
\int_{\mr^n}\sum^{n}_{i=1} u^{T}  \cdot (\frac{\pa \alpha_{i}}{\pa x_{i}}) dx=-\int_{\mr^n} \sum^{n}_{i=1} f_{i}^{T} \cdot \alpha_{i} dx,
$$
holds for any $\alpha=\sum^{n}_{i=1} \alpha_{i}dx_i \in \Lambda^1(D, E)$ with compact support.
A simple calculation shows that
$$
\langle \langle  du,\alpha \rangle \rangle_g =\langle  \langle  u, -\sum^{n}_{i=1}(g^{-1}\frac{\pa g}{\pa x_{i}} \alpha_{i}+\frac{\pa \alpha_{i}}{\pa x_{i}})  \rangle  \rangle_g,
$$
which means that the formal adjoint of the operator $d:L^2(D, E)\ra L^2_1(D, E)$ is given by
\begin{equation}
d^{*} \alpha=-\sum^{n}_{i=1}(g^{-1}\frac{\pa g}{\pa x_{i}} \alpha_{i}+\frac{\pa \alpha_{i}}{\pa x_{i}}).
\end{equation}

Now we can give the proof of Proposition \ref{prop:Bochner type identity}, which is restated as follows.

\begin{prop}[=Proposition \ref{prop:Bochner type identity}]
Assume $D$ is a domain in $\mr^n$, $E=D \times \mr^r$ is the trivial vector bundle defined over $D$, and $g: D \ra  Sym_{r}(\mr)^{+}$ is a $C^2$ Riemannian metric on $E$.
Then for any $\alpha =\sum^{n}_{i=1} \alpha_{i} dx_{i}\in \Lambda^1(D, E)$ with compact support, we have
\begin{equation}\label{eq:Bochner-type-identity-matrix}
\int_{D} \langle  d^*\alpha,d^* \alpha \rangle _{g}dx+\int_{D} \langle  d\alpha,d \alpha \rangle _{g}dx=\int_{D}\sum^{n}_{i,j=1}\langle  \theta^{g}_{ij}\alpha_{i},\alpha_{j} \rangle  _{g}dx+\int_{D} \sum^{n}_{i,j=1} \langle   \frac{\pa \alpha_{i}}{\pa x_{j}}, \frac{\pa \alpha_{i}}{\pa x_{j}} \rangle _{g} dx,
\end{equation}
where $d^*$ is the formal adjoint operator of $d:L^2(D,E)\ra L^2_1(D,E)$.
\end{prop}

\begin{proof}
\par Since $\alpha $ is compactly supported, we have
$$
\int_{D} \langle  d^*\alpha,d^* \alpha \rangle _{g}dx=\int_{D} \langle \alpha,dd^* \alpha \rangle _{g}dx.
$$
We know
$$
d^{*} \alpha=-\sum^{n}_{i=1}(g^{-1}\frac{\pa g}{\pa x_{i}} \alpha_{i}+\frac{\pa \alpha_{i}}{\pa x_{i}})=-\sum^{n}_{i=1} \delta_{i} \alpha_{i},
$$
where
$$
\delta_{i}=g^{-1}\frac{\pa g}{\pa x_{i}} \cdot + \frac{\pa }{\pa x_{i}},
$$
thus
$$
dd^* \alpha =-\sum^{n}_{j=1}\sum^{n}_{i=1} \frac{\pa}{\pa x_{j}} (\delta_{i} \alpha_{i})dx_{j}.
$$
Note that
$$
\frac{\pa}{\pa x_{j}} \delta_{i} \alpha_{i}=\frac{\pa}{\pa x_{j}}(g^{-1}\frac{\pa g}{\pa x_{i}})\alpha_{i}+g^{-1}\frac{\pa g}{\pa x_{i}} \frac{\pa \alpha_{i}}{\pa x_{j}}+\frac{\pa^2 \alpha_{i}}{\pa x_{i}\pa x_{j}}
$$
and
$$
\delta_{i} \frac{\pa \alpha_{i}}{\pa x_{j}}=g^{-1}\frac{\pa g}{\pa x_{i}} \frac{\pa \alpha_{i}}{\pa x_{j}}+\frac{\pa^2 \alpha_{i}}{\pa x_{i}\pa x_{j}},
$$
thus we have
$$
dd^* \alpha =-\sum^{n}_{i,j=1} \frac{\pa}{\pa x_{j}}(g^{-1}\frac{\pa g}{\pa x_{i}})\alpha_{i}dx_{j}-\sum^{n}_{i,j=1} \delta_{i} \frac{\pa \alpha_{i}}{\pa x_{j}}dx_{j}.
$$
Since
$$
\int_{D} \langle  \alpha_{k}, \delta_{i}\alpha_{i} \rangle _{g} dx=-\int_{D} \langle  \frac{\pa}{\pa x_{i}}\alpha_{k}, \alpha_{i} \rangle _{g} dx,
$$
 we have
$$
\begin{aligned}
\int_{D} \langle \alpha,dd^* \alpha \rangle _{g}dx &=\int_{D}\sum^{n}_{i,j=1}\langle  \theta^{g}_{ij}\alpha_{i},\alpha_{j} \rangle  _{g}dx-\int_{D} \sum^{n}_{i,j=1} \langle  \delta_{i} \frac{\pa \alpha_{i}}{\pa x_{j}}, \alpha_{j} \rangle _{g} dx \\
&=\int_{D}\sum^{n}_{i,j=1}\langle  \theta^{g}_{ij}\alpha_{i},\alpha_{j} \rangle  _{g}dx+\int_{D} \sum^{n}_{i,j=1} \langle   \frac{\pa \alpha_{i}}{\pa x_{j}}, \frac{\pa \alpha_{j}}{\pa x_{i}} \rangle _{g} dx.
\end{aligned}
$$
Note that
$$
\begin{aligned}
\int_{D} \sum^{n}_{i,j=1} \langle   \frac{\pa \alpha_{i}}{\pa x_{j}}, \frac{\pa \alpha_{j}}{\pa x_{i}} \rangle  _{g} dx &=\int_{D} -\half \sum^{n}_{i,j=1} || \frac{\pa \alpha_{k}}{\pa x_{j}}-\frac{\pa \alpha_{j}}{\pa x_{k}}||^{2}_{g}+ \sum^{n}_{i,j=1} \langle \frac{\pa \alpha_{i}}{\pa x_{j}},\frac{\pa \alpha_{i}}{\pa x_{j}} \rangle_{g} dx \\
&=-\int_{D} \langle d\alpha,d \alpha \rangle_{g} dx+\int_{D}\sum^{n}_{i,j=1} \langle   \frac{\pa \alpha_{i}}{\pa x_{j}}, \frac{\pa \alpha_{i}}{\pa x_{j}} \rangle _{g} dx,
\end{aligned}
$$
we have
\begin{equation*}
\int_{D} \langle  d^*\alpha,d^* \alpha \rangle _{g}dx+\int_{D} \langle  d\alpha,d \alpha \rangle _{g}dx=\int_{D}\sum^{n}_{i,j=1}\langle  \theta^{g}_{ij}\alpha_{i},\alpha_{j} \rangle  _{g}dx+\int_{D} \sum^{n}_{i,j=1} \langle   \frac{\pa \alpha_{i}}{\pa x_{j}}, \frac{\pa \alpha_{i}}{\pa x_{j}} \rangle _{g} dx,
\end{equation*}
which is that we want to prove.
\end{proof}

\section{Characterization of Nakano positivity of Riemannian vector bundles}\label{sec:cha of Nakano positivity}
In this section, we give the proof of Theorem \ref{thm:cha nakano positive}.
\begin{thm}[=Theorem \ref{thm:cha nakano positive} ]
Assume $D$ is a domain in $\mr^n$, $E=D \times \mr^r$ is the trivial vector bundle  over $D$,
and $g: D \ra  Sym_{r}(\mr)^{+}$ is a Riemann metric on $E$.
If for any smooth strictly convex function $\psi$ on $D$ and any $d$-closed $1$-form $f \in \Lambda^1(D, E)$ with compact support,
the equation $du=f$ can be solved with $u \in L^2 (D, E)$ satisfying the following estimate
$$
\int_{D} \langle u,u \rangle  _{g} e^{-\psi} dx \leq  \int_{D}\langle (Hess\; \psi)^{-1}f,f \rangle  _{g} e^{-\psi} dx,
$$
 then $(E, g)$ is Nakano semipositive.
\end{thm}
\begin{proof}
We follow the spirit of the proof of Theorem \ref{thm:char convexity}.

Let $\psi$ be a smooth strictly convex function on $D$, and $f=\sum^n_{j=1}f_idx_i$ be an arbitrary element in $\Lambda^1(D,E)$ with compact support.
By assumption, we can solve the equation $du=f$, with the estimate
$$
\int_{D} \langle u,u \rangle  _{g} e^{-\psi} dx \leq  \int_{D}\langle (Hess\; \psi)^{-1}f,f \rangle  _{g} e^{-\psi} dx.
$$

For any $\alpha=\sum^n_{i=1} \alpha_idx_i\in \Lambda^0(D, E)$, we have
$$
\begin{aligned}
 \langle  \langle \alpha ,f \rangle   \rangle_g^2 &= \langle  \langle \alpha ,du \rangle   \rangle_g^2=\langle   \langle   d^*\alpha ,u \rangle  \rangle_g^2 \\
&\leq \langle   \langle   d^*\alpha ,d^*\alpha \rangle  \rangle_g  \cdot  \langle   \langle   u ,u \rangle  \rangle_g  .
\end{aligned}
$$
Combing this with Proposition \ref{prop:Bochner type identity}, we get
$$
\begin{aligned}
\langle   \langle  \alpha ,f \rangle  \rangle_g^2 & \leq   \int_{D}\langle  (Hess\; \psi)^{-1}f,f \rangle _{g} e^{-\psi} dx \\
& \times \left(\int_{D} \sum^{n}_{j,k=1} \langle    \frac{\pa \alpha_{j}}{\pa x_{k}}, \frac{\pa \alpha_{k}}{\pa x_{j}} \rangle _{g}  e^{-\psi}dx+\int_{D} \sum^{n}_{j,k=1} \langle  (\theta^{g}_{jk}+(Hess\; \psi))\alpha_{j},\alpha_{k} \rangle _{g} e^{-\psi}dx\right),
\end{aligned}
$$
Setting $\alpha=(Hess\; \psi)^{-1}f,$ then we obtain the following inequality:
\begin{equation}\label{main-inequality-matrix}
\int_{D} \sum^{n}_{j,k=1} \langle  \theta^{g}_{jk}\alpha_{j},\alpha_{k} \rangle _{g} e^{-\psi}dx +\int_{D} \sum^{n}_{j,k=1} \langle    \frac{\pa \alpha_{j}}{\pa x_{k}}, \frac{\pa \alpha_{k}}{\pa x_{j}} \rangle _{g} e^{-\psi} dx \geq 0.
\end{equation}

 Next we argue by contradiction. Suppose $(E, g)$ is not Nakano semipositive, then there exist $x_{0} \in D$, a constant $a>0$,
 and an $1$-form $\displaystyle \xi=\sum^{n}_{i=1}\xi_{i} dx_{i}\in\Lambda^1(D,E)$ with constant coefficients $\xi_{i} \in \mr^r$ such that
$$
\sum^{n}_{j,k=1} \langle \theta^{g}_{jk}(x)\xi_{j},\xi_{k} \rangle _{g}<-c
$$
holds for any $x \in B(x_0 ,a ) \subseteq  D$, where $c>0$ is a constant.
For simplicity, we assume $x_0 =0$, and denote $B(0,a)$ by $B_{a}$.

Choose $\chi(x)\in C^{\infty}(\mr^n)$ with support in $D_a$, such that $\chi(x)=1$ for $x \in B_{\frac{a}{2}}$.
Let
$$
v(x) =(\sum^{n}_{i=1} \xi_{i} x_{i})\chi(x) ,
$$
which is viewed as an element in $\Lambda^0(D, E)$ with compact support.
Let $f=d v\in \Lambda^0(D, E)$, then $f$ has compact support and $df=0$.
On $B_{\frac{a}{2}}$, we have
$$
f=\sum^{n}_{i=1}\xi_{i} dx_{i}.
$$

For any $s>0$, we set $\psi_{s}(x)=s(|x|^2-\frac{a^2}{4})$,
then $\psi_s$ is a smooth strictly convex function on $\mr^n$, and the Hessian of $\psi_s$ is given by $Hess\psi_s=2s I_n$,
where $I_n$ is the unit matrix of order $n$.
We set $\alpha_{s}=(Hess\; \psi_{s})^{-1}f=\frac{1}{2s}f\in \Lambda^1(D,E)$.
Since $f$ has compact support, there exists a constant $C>0$ such that $|\frac{\pa(\alpha_{s})_{k}}{\pa x_{j}}|  \leq \frac{C}{s}$ hold for any $x \in D, 1\leq j,k \leq n, \forall s >0$.
By construction, we also have $|\frac{\pa(\alpha_{s})_{k}}{\pa x_{j}}| =0$ on $B_{\frac{a}{2}}$.
We now give an estimate of the left hand side of the inequality (\ref{main-inequality-matrix}), with $\alpha$ and $\psi$ replaced by $\alpha_{s}$ and $\psi_{s}$, and multiplied by $s^2$:
$$
\begin{aligned}
& s^2\int_{D}\sum^{n}_{j,k=1} \langle \theta^{g}_{jk} (\alpha_{s})_{j},(\alpha_{s})_{k} \rangle _{g}e^{-\psi_{s}}dx +s^2\int_D\sum^{n}_{j,k=1} \langle   \frac{\pa (\alpha_{s})_{j}}{\pa x_{k}}, \frac{\pa (\alpha_{s})_{k}}{\pa x_{j}} \rangle _{g} e^{-\psi_{s}}dx \\
\leq &   -c \int_{B_{\frac{a}{2}}} e^{-\psi_{s}}dx+\frac{1}{4}\int_{D\backslash B_{\frac{a}{2}}}\sum^{n}_{j,k=1} \langle \theta^{g}_{jk}f_{j},f_{k} \rangle _{g} e^{-\psi_{s}}dx+C^2\int_{D\backslash B_{\frac{a}{2}} }\sum^{n}_{j,k=1} e^{-\psi_{s}}dx.
\end{aligned}
$$
Since $\underset{s \ra + \infty}{lim} \psi_{s} (x) =+\infty, \forall x \in D \backslash B_{\frac{a}{2}}$, and $\psi_{m}(x) \leq 0,\forall x \in B_{\frac{a}{2}}$ and $\forall s >0$,
we get
 $$
 \int_{D} \sum^{n}_{j,k=1}  \langle \frac{\pa (\alpha_{s})_{j}}{\pa x_{k}}, \frac{\pa (\alpha_{s})_{k}}{\pa x_{j}} \rangle _{g} e^{-\psi_{s}} dx+\int_{D} \sum^{n}_{j,k=1} \langle \theta^{g}_{jk}(\alpha_{s})_{j},(\alpha_{s})_{k} \rangle _{g} e^{-\psi_{s}}dx <0
 $$
for $s$ sufficiently large, which contradicts to the inequality (\ref{main-inequality-matrix}).
Thus $(E, g)$ must be Nakano semipositive.

 \end{proof}

\section{Prekopa's Theorem for matrix valued functions}\label{sec:matrix valued prekopa}
In this section we give a new proof of Prekopa's Theorem for matrix valued functions, namely Theorem \ref{thm:Prekopa-matrix-valued}.
The proof is based on an combination of Theorem \ref{thm:cha nakano positive} and Theorem \ref{thm:L^2 estimate for vector bundle}.

\begin{thm}[=Theorem \ref{thm:Prekopa-matrix-valued}]
Let $\tilde g(x,y):\mr^{n}_{x} \times \mr^{m}_{y} \ra Sym_{r}(\mr)^+$ be a $C^2$ Riemannnian metric on the trivial bundle $\tilde E=(\mr^n\times\mr^m)\times\mr^r\ra \mr^n\times\mr^m$ over $\mr^n\times\mr^m$.
Define
$$
g(x)=\int_{\mr^{m}} \tilde g(x,y) dy \in Sym_{r}(\mr)^+.
$$
If $(\tilde E, \tilde g)$ is Nakano semipositve and $\tilde g$ is $C^2$ smooth, then $ g$ is Nakano semipositive, viewed as a Riemannian metric on the trivial bundle $E=\mr^n\times\mr^r$ over $\mr^n$.
\end{thm}
\begin{proof}
 Let $f=\sum^m_{j=1}f_j(x)dx_j\in\Lambda^1(\mr^n,E)$ be a compactly supported $d$-closed 1-form on $\mr^n$ with values in $E$.
 We can naturally view $f$ as a $d$-closed 1-form on $\mr^n\times\mr^m$ with values in $\tilde E$ which will be denoted by $\tilde f=\sum^{n}_{j=1}\tilde f_jdx_j$.
 Let $\psi(x)$ be an arbitrary smooth strictly convex function on $\mr^n$.
 We also view $\psi$ as a smooth function on $\mr^n\times\mr^m$ which will be denoted by $\tilde \psi$.

According to Theorem \ref{thm:L^2 estimate for vector bundle}, the equation $d\tilde u=\tilde f$ can be solved on $\mr_x^n\times\mr_y^m$ with
$\tilde u\in L^2(\mr^n\times\mr^m, \tilde E)$ satisfying the estimate
$$
\int_{\mr^n \times \mr^m}|\tilde u|_{\tilde g}^2 e^{-\tilde \psi} dxdy\leq \int_{\mr^n \times \mr^m} \sum_{i,j=1}^n \langle \tilde \psi ^{ij}\tilde f_i, \tilde f_j\rangle_{\tilde g} e^{-\tilde \psi}dxdy,
$$
where $(\tilde\psi^{ij}(x,y))_{n\times n}$ is the inverse of the matrix $(\frac{\partial^2\tilde\psi}{\partial x_i\partial x_j})_{n\times n}$.

From the identity $d\tilde u=\tilde f$, we see that $\tilde u$ is independent of $y_1,\cdots, y_m$,
hence we can view $\tilde u$ as an element in $L^2(\mr^n,E)$, which will be denoted by $u$.
It is clear that $du=f$.
By Fubini's theorem, the above inequality becomes to
$$
\int _{\mr^n }|u|_g^2 e^{-\psi} dx \leq \int _{\mr^m }\sum_{i,j}  \langle\psi ^{ij} f_i, f_j\rangle_g  e^{-\psi}dx=\int_{\mr^n}\langle (Hess\psi)^{-1}f, f\rangle_ge^{-\psi},
$$
which implies, by Theorem \ref{thm:cha nakano positive}, that $(E,g)$ is Nakano semipositive.
\end{proof}

%

\bibliographystyle{amsplain}

\end{document}